\documentclass[11pt]{amsart}

\usepackage{amsmath}
\usepackage{a4wide}

\newcommand{\N}{{\mathbb N}}
 
\newcommand{\R}{{\mathbb R}}

\newcommand{\Per}{\mathrm{Per}}

\newtheorem{theorem}{Theorem}[section]

\newtheorem{lemma}[theorem]{Lemma}
\newtheorem{proposition}[theorem]{Proposition}

\theoremstyle{definition}
\newtheorem{definition}[theorem]{Definition}
\newtheorem{remark}[theorem]{Remark}

\newcommand{\eps}{\varepsilon}

\begin{document}
 
\title{Nonlocal minimal clusters in the plane}

% Place all authors' names in [ ] shown as running head, Leave { } empty
% Please use `and' to connect the last two names if applicable
% Use FirstNameInitial.  MiddleNameInitial. LastName, or only last names of authors if there are too many authors
\author{Annalisa Cesaroni}
\address{Dipartimento di Scienze Statistiche,
Universit\`{a} di Padova, Via Cesare Battisti 241/243, 35121 Padova, Italy}
\email{annalisa.cesaroni@unipd.it}
%\author{Inwon Kim}
\author{Matteo Novaga}
\address{Dipartimento di Matematica, Universit\`{a} di Pisa, Largo Bruno Pontecorvo 5, 56127 Pisa, Italy}
\email{matteo.novaga@unipi.it}

% It is required to enter 2010 MSC.
% Please provide minimum  5 keywords.
 
\begin{abstract} 
We prove existence of partitions of an open set $\Omega$ with a given number of  phases, which minimize  the sum of the fractional perimeters of all the phases, with Dirichlet boundary conditions.  
In two dimensions we show that, if  the fractional parameter $s$ is sufficiently close to $1$,   the only singular minimal cone, that is, the only minimal partition  invariant by dilations and with a singular point,  is given by three half-lines meeting at $120$  degrees. In the case of a weighted sum of  fractional perimeters, we show that there exists a unique minimal cone with three phases. %whose opening angles are uniquely defined by the weights.  
\end{abstract} 

\maketitle 

\tableofcontents

\section{Introduction}
A $k$-cluster  is a family  $\mathcal{E}=(E_i)_{i=1,\dots k}$ of disjoint measurable  subsets of $\R^d$ such that $\cup_i E_i=\R^d$, up to a negligible set. 
We call each set $E_i$ a phase of the cluster. 
Following \cite{cm}, for an open set $\Omega\subset\R^d$ and $s\in (0,1)$ 
we  define the fractional perimeter of $\mathcal{E}$ 
relative to $\Omega$  as  
\begin{equation}\label{fun2} \mathcal{P}_s(\mathcal{E}; \Omega):=\sum_{1\leq i\leq k} \Per_s (E_i; \Omega),
\end{equation} 
where
%The nonlocal perimeter relative to $\Omega$  for a set  $E\subseteq \R^d$ is the nonlocal  interaction between $E$ and its complement $\R^d\setminus E$   relative to $\Omega$:  for $s\in (0,1)$
\begin{eqnarray}\label{fractionalp}
 \Per_s(E;\Omega)&:=&J_s(E\cap \Omega, \R^d\setminus E)+J_s(\Omega\setminus E, E\setminus \Omega)
 \qquad \text{for}\ E\subset \R^d,
\\ \nonumber
J_s(A,B)&:=&\int_{A}\int_{B}\frac{1}{|x-y|^{d+s}}dxdy \qquad \text{for}\ A,B\subset \R^d,\ |A\cap B|=0.
\end{eqnarray}

The functional in \eqref{fun2} and more generally the weighted fractional perimeter
\begin{equation}\label{wp}
\mathcal{P}_{s, c}(\mathcal{E};\Omega):=\sum_{1\leq i\leq k} c_i\Per_s (E_i; \Omega),
\end{equation} 
with $c=(c_i)_i$ and $c_i>0$, are a natural generalization of the (weighted) classical perimeter of a cluster
\begin{equation}\label{wpc}\mathcal{P}_{c}(\mathcal{E};\Omega):=\sum_{1\leq i\leq k} c_i\Per (E_i; \Omega),
\end{equation}
  and arise  in the analysis of equilibria for a mixture of $k$  immiscible fluids in a container $\Omega$, where  the fluids tend to occupy disjoint regions  in such a way  to minimize the total surface tension measured through nonlocal interaction energies, rather then through surface area as in the classical case.  

%We point out  that in our  formulation the number of phases $k$ is fixed, and it is not part of the minimization problem.  
%We observe moreover that if $k=2$, the problem is functional \eqref{fun2} reduces to the fractional perimeter of a single set, since $E_2=\R^d\setminus E_1$. 

In \cite{cm} the authors proved existence of fractional isoperimetric  
clusters. More precisely, they showed that there exists a minimizer of the energy \eqref{fun2} with $\Omega=\R^d$, among all $k$-clusters such that 
each phase has a prescribed volume.
They also established the regularity of such minimal clusters, showing that the singular set has  Hausdorff dimension less than $d-2$ (and it is discrete in the planar case $d=2$), that  outside from the singular set the boundary of the cluster is a hypersurface of class $C^{1,\alpha}$ 
for some $\alpha>0$, and finally that the blow-up of the cluster at a singular point is a minimal cone.  

In this short note we consider minimizers of \eqref{wp} in a bounded open set $\Omega\subset\R^d$, with
Dirichlet data.  More precisely, we fix the phases $E_i$ outside $\Omega$, that is, we fix exterior data
\begin{equation}\label{boundary} 
(\bar E_1, \bar E_2, \dots , \bar E_k)\qquad \bar E_i\subseteq \R^d\setminus \Omega, \forall i\quad \cup_i \bar E_i=\R^d\setminus \Omega,
\end{equation}
and we show existence of a solution to  the following Dirichlet problem \begin{equation}\label{dir2} 
\inf_{\{\mathcal{E},\  E_i\setminus \Omega=\bar E_i\}} \mathcal{P}_{s,c}(\mathcal{E};\Omega)
\end{equation} for $c=(c_i)_i$, with $c_i>0$. 

We are particularly interested in the analysis of singularities  in dimension $d=2$, in order to characterize fractional clusters in some basic cases. 
For instance, in Theorem \ref{cone} we consider the energy \eqref{fun2} and we show that for $s$ sufficiently close to $1$,   the only singular minimal cone consists of three half-lines meeting at $120$  degrees at a common end-point.  In particular, this implies that  the unique local minimizers for the fractional perimeter on $k$-clusters, for $s$ sufficiently close to $1$, are half-planes and such singular $3$-cones. 
We recall that, for $k=2$, half-planes are the unique local minimizers for any $s\in (0,1)$, as proved in \cite{adpm,crs} (see also \cite{c,p} for the extension to more general energies). 

To obtain our result,  we first provide  the $\Gamma$-convergence of the fractional perimeter of  a $k$-cluster to the classical perimeter as $s\to 1$,  which is a generalization of the analogous result proven in \cite{adpm,cv1} for $k=2$, and the Hausdorff convergence of minimizers which  is obtained  by exploiting the density estimates obtained in \cite{cm}.  We also show that this convergence can be improved outside the singular set.

Finally,  we consider the analogous problem  for   weighted fractional perimeters, restricted to  $3$-clusters. In Proposition \ref{cone2} we show that there exists a unique minimal $3$-cone, whose opening angles are uniquely determined in terms of the weights $c_i$. 
%We also discuss briefly the case in which one of the weights is nonpositive. 

\subsection*{Acknowledgements}
The authors wish to thank Valerio Pagliari and Alessandra Pluda for useful discussions on the
topic of this paper. The authors are members of the INDAM/GNAMPA.
 
 \section{The Dirichlet problem} 
We start proving existence of minimizers of problem \eqref{dir2}, then we discuss the regularity of solutions, and finally the convergence  of the minimizers as $s\to 1 $ to the solution of the analogous Dirichlet problem for the classical perimeter. 

 \begin{theorem}\label{exth}
Let $\Omega\subseteq \R^d$ be an open bounded set of finite perimeter and fix  an exterior datum as in \eqref{boundary}.  Then,  there exists a solution to the Dirichlet problem \eqref{dir2}.

\end{theorem}  

\begin{proof} First  of all note that if we consider $\mathcal{E}$ defined  as follows:
$E_1= \Omega\cup \bar E_1$, $E_j=\bar E_j$ for $j\neq 1$, then 
we get $\mathcal{P}_s(\mathcal{E};\Omega)\leq  k \max c_i \Per_s(\Omega)<+\infty$ for all $j$, since $\Omega$ is bounded of finite perimeter (see \cite{cv1}).
The existence result is then obtained by the direct method of the calculus of variations, using the fact that $\Per_s(E)$ is a Gagliardo norm of $\chi_E$, recalling   that a uniform bound on the Gagliardo norm implies compactness in $L^2$, and that the norm is lower semicontinuous 
with respect to the $L^1$-convergence (see \cite{sv}). 
\end{proof} 

We recall the density estimates proved in \cite{cm}, which are uniform with respect to $s\to 1$.    
%We state them  for  minimizers of the Dirichlet problem \eqref{dir2}. 

\begin{theorem}[Density estimates] \label{density}
Let $s_0\in (0,1)$,  and let $\mathcal{E}$ be a minimizer of \eqref{dir2} for some $s\in [s_0,1)$.    
Then there exist  $\sigma_0= \sigma_0(d,s_0,c), \sigma_1= \sigma_1(d,s_0,c)\in (0,1)$ such that,   if  $x\in \partial E_i\cap \Omega$ for some $i$, then 
\[ \sigma_0\omega_d r^d\leq|E_i\cap B(x, r)|\leq \sigma_1\omega_d r^d\qquad \forall r< d(x, \partial \Omega). \]
 \end{theorem} 
 
 \begin{proof} 
 The proof can be obtained as a straightforward adaptation of the proof of Lemma 3.4, the infiltration lemma,  in \cite{cm}. 
 We note that if we fix $x\in \Omega$, then $\mathcal{E}$ is a $(\Lambda, d(x, \partial\Omega))$ minimizer for every $\Lambda>0$ and  observing  in the proof that  the constant $r_1$ can be chosen equal to $r_0$ and that  $\sigma_0$ is  uniform as $s\to 1$. 
 \end{proof}
 
 \begin{remark}\label{remdensity} 
 By inspecting the proof of \cite[Lemma 3.4]{cm}, we get that  this estimate degenerates  as $s\to 0$, in fact  $\lim_{s_0\to 0^+}\sigma_0(d,s_0)= 0$. 
 \end{remark} 
 
Let us fix a partition $\mathcal{E}$ and a point $x\in \partial \mathcal{E}$. 
The blow-up of $\mathcal{E}$  at $x$  is the cluster $\mathcal{E}_{x,r}$  defined by
\[E^{x,r}_i= \frac{E_i-x}{r}.\] 
We state the regularity  result in \cite[Theorem 3.3, Theorem 3.7]{cm}, adapted to our problem, with an improvement of the regularity given by the application of a bootstrap argument given in \cite{bfv}. We note that the same argument   also applies to   the isoperimetric clusters considered in \cite{cm}, 
and allows to improve the regularity of the boundary outside the singular set from $C^{1, \alpha}$ to $C^\infty$. 

We first  recall the definition of cone, and of regular and singular points. 

 \begin{definition}  A partition $\mathcal{C}$ is called a $k$-cone with vertex $x_0$ if  it  is invariant by dilatation, that is $\lambda (\mathcal{C}-x_0)=\mathcal{C}-x_0$ for every $\lambda>0$, and it has $k$-phases $C_1, \dots, C_k$. \end{definition}
\begin{definition}\label{regpo} Let $\mathcal{E}$ be a $k$-cluster. $x\in \partial\mathcal{E}$ is a regular point if there exist an half-space $H$ and  two indexes $i,j$, such that  as $r\to 0$
 \[E^{x,r}_i\to H, \quad E^{x,r}_j\to \R^d\setminus H,\quad E^{x,r}_h\to \emptyset \text{ for $h\neq i,j$},  \]
locally in $L^1(\R^d)$.  The set of regular points will be denoted by $\mathcal{R}(\mathcal{E})$, 
while the complementary set $\partial \mathcal{E}\setminus \mathcal{R}(\mathcal{E})$ will be called singular set.
\end{definition}

\begin{theorem}\label{thmsing}  
Let $\Omega\subseteq\R^d$  be an open set and $\mathcal{E}$ be a $k$-cluster which is a solution to the Dirichlet problem \eqref{dir2}, with a given boundary datum as in \eqref{boundary}. For every  $x\in \partial \mathcal{E}\cap \Omega$, there exist a $h$-cone $\mathcal{C}$, with $h\leq k$, 
and a sequence $r_j\to 0$,
such that
 
 \[
 \lim_{j\to +\infty}E_i^{x,r_j} = C_i \text{  in  $L^1_{loc}(\R^d)$ and locally uniformly, }\qquad \forall i=1,\dots, h.\]
Moreover, the singular set $(\partial \mathcal{E}\setminus \mathcal{R}(\mathcal{E}))\cap \Omega$ is  (relatively)  closed, of Hausdorff dimension less than $d-2$ and discrete if $d=2$. Finally for every $x\in  \mathcal{R}( \mathcal{E})\cap \Omega$ there is $r_x>0$ such that 
$\partial\mathcal{E}\cap B(x, r_x)$ is a $C^\infty$ hypersurface in $\R^d$.
\end{theorem}

\begin{proof}
The first part of the statement about the convergence of the blow up can be obtained by a direct adaptation  of the proof of the analogous theorem  given in \cite[Theorem 3.7]{cm}. 
As for the dimension of the singular set,  it can be obtained exactly as in \cite[Theorem  3.13, Proposition 3.14]{cm}. 

Fix now $x\in  \mathcal{R}(\mathcal{E})\cap \Omega $. Proceeding as in \cite[Theorem 3.3]{cm},  by exploiting the definition of regular point and the density estimates, we get that there exist $r>0$ and 
 two indexes $i,j$ and such that $E_h\cap B(x,2r))=\emptyset$ for every $h\neq i,j$ and $B(x,2r)\subset \Omega$.  
 We observe that there exists $\Lambda>0$, depending on $c, r, s$, such that   $E_i$ is a $\Lambda$-minimizer for $\Per_s$ in $B(x,r)$  in the following sense:  \begin{equation}\label{qm}\Per_s(E_i; B(x,r))\leq \Per_s(F; B(x,r) +\Lambda |E_i\Delta F|\qquad \forall F\subseteq \R^d,\ F\Delta E_i\subseteq B(x,r).\end{equation}
 This property is easily checked using the fact that $\mathcal{E}$ is a solution to the Dirichlet problem. Indeed we define  the $k$-cluster 
$\mathcal{F}^{ij}$ in this way: $F^{ij}_i=(E_i\setminus B(x,r) )\cup (F\cap B(x,r))$, $F^{ij}_j=(E_j\setminus B(x,r))\cup (B(x,r)\setminus F)$ and $F^{ij}_h=E_h$ for all $h\neq i,j$. Note that $\mathcal{F}^{ij}$ satisfies the same boundary conditions as $\mathcal{E}$. Following  \cite[Theorem 3.3]{cm},
and recalling that $B(x, r)\setminus E_j= E_i\cap B(x,r)$ and that $E_h\cap B(x,2r)=\emptyset$ for all $h\neq i,j$,  
an easy computation gives
\begin{align*} 0\leq& \mathcal{P}_{s,c}(\mathcal{F}^{ij}; \Omega)-\mathcal{P}_{s,c}(\mathcal{E}; \Omega)= (c_i+c_j)\Per_s(F; B(x,r))- (c_i+c_j)\Per_s(E_i; B(x,r))\\ & +c_jJ_s(E_i\cap B(x,r), (\R^d\setminus (E_i\cup E_j))-c_jJ_s(F\cap B(x,r), (\R^d\setminus (E_i\cup E_j))\\
\leq &   (c_i+c_j)\Per_s(F; B(x,r))-  (c_i+c_j)\Per_s(E_i; B(x,r))+c_jJ_s(E_i\setminus F, \R^d\setminus B(x,2r))
\\
\leq &   (c_i+c_j)\left[\Per_s(F; B(x,r))-  \Per_s(E_i; B(x,r))+\frac{c_j}{c_i+c_j}\frac{d r^s\omega_d}{s}|E_i\Delta F|\right]. 
\end{align*}  

Using this minimality property we may conclude exactly as in \cite[Theorem 3.3]{cm} that, possibly reducing $r$, all the points in $\partial E_i\cap B(x,r)$  are regular and that  $\partial E_i\cap B(x,r)$ is a
$C^{1, \alpha}$ hypersurface, for some $\alpha$ depending on $s$.

Finally,  at the regular points of $\partial\mathcal{E}\cap \Omega$ we can write the Euler-Lagrange equation.  Let $x$ and $i,j$ as before. Then it is possible to show that there exists a constant $c_{ij}$ such that the stationarity condition at $x$ reads 
\[c_i H_s(x, E_i)- c_j H_s(x, E_j)=c_{ij}\] where, for $E\subseteq \R^d$ and $x\in \partial E$, $H_s(x,E)$ is the fractional curvature, defined as   
\begin{equation}\label{fcurv} 
H_s(x, E)=   \int_{\R^d } \frac{\chi_{\R^d\setminus E}(y)-\chi_{E}(y)}{|x-y|^{d+s}}dy. 
\end{equation}  Exploiting this definition, we obtain that the stationary condition can be written as the
following equation
\[(c_i+c_j) H_s(x, E_i)=c_{ij} +2c_j\int_{\R^d\setminus(E_i\cup E_j)} \frac{1}{|x-y|^{d+s}}dy, \]
which holds in the viscosity sense.  We note that if $y\in\R^d\setminus(E_i\cup E_j)$, then $|x-y|\geq 2r > 0$, so that the r.h.s. is a smooth function of $x$. We apply now the bootstrap argument in \cite[Theorem 5]{bfv}  to conclude that   $\partial E_i\cap B(x,r)$ is a
$C^\infty$ hypersurface. 
\end{proof}

We now recall a density result of polyhedral clusters with respect to the (weighted) local perimeter \eqref{wpc}, which has been obtained in \cite{bcg}. We 
shall adapt this result in order to apply it to Dirichlet problems. In particular, we will need the notion of transversality of a cluster, 
which ensures that the polyhedral approximations can be chosen also to fit with the exterior data, up to a small error.

\begin{definition}[Polyhedral clusters]  
A $k$-cluster $\mathcal{K}=(K_i)_{i=1,\dots, k}$ is polyhedral in an open set $\Omega$ if  for every phase $K_i$ there is a finite number of $(d-1)$-dimensional simplexes  
 $T_1, \dots, T_{r_i}\subseteq \R^d$  such that $\partial K_i$ coincides, up to a $\mathcal{H}^{d-1}$-null set, with   $\cup_jT_j\cap\Omega$. 
\end{definition}

\begin{definition} \label{trdef}
Let $\Omega$ be an open set of class $C^1$. For $\delta>0$ we define 
\[\Omega^\delta:=\{x\in\R^d \ :\ d(x, \Omega) < \delta\}\qquad \Omega_\delta:=\{x\in\Omega\ :\ d(x, \R^d\setminus\Omega)>\delta\}.\] 
We say that a measurable  set  $F$ is transversal to $\partial \Omega$ if 
\[\lim_{\delta\to 0^+}\Per(F; \Omega^\delta\setminus \Omega_\delta)=0.\] 
We say that  $F$ is transversal to $\partial \Omega^+$ if 
\[\lim_{\delta\to 0^+}\Per(F; \Omega^\delta\setminus \Omega)=0.\] 
A cluster is transversal to $\partial \Omega$ (resp. to $\partial \Omega^+$) if every phase is transversal. 
\end{definition}  
  
 \begin{theorem} \label{dens} Let $\Omega$ be a bounded open  set with $C^1$ boundary, and let $\mathcal{F}$ be a cluster in $\Omega$ such that every phase $F_i$ has finite perimeter in $\Omega$. 
For every $\eps>0$ there exists  a cluster $\mathcal{K}_\eps$  which is polyhedral in $\Omega$, such that $\mathcal{K}_\eps\to \mathcal{F}$ 
in $L^1(\Omega)$ and $\mathcal{P}_{c}(\mathcal{K}_\eps;\Omega)\to \mathcal{P}_{c}(\mathcal{F};\Omega)$. 

Assume moreover  that   $\mathcal{F}$  is polyhedral in $\R^d\setminus \Omega$ and transversal to $\partial \Omega^+$. Then for every $\eps>0$ there exists  a polyhedral cluster $\mathcal{K}_\eps$ with the following properties:
\begin{itemize}
\item[i)] $\mathcal{K}_\eps\to \mathcal{F}$ in $L^1(\Omega)$,\item[ii)]   $\mathcal{K}_\eps=\mathcal{F}$ in $\R^d\setminus \Omega$, \item[iii)] $\mathcal{K}_\eps$  is transversal to $\partial \Omega$, 
\item[iv)]   
$\mathcal{P}_{c}(\mathcal{K}_{\eps}; \Omega)\to \mathcal{P}_{c}(\mathcal{F}; \Omega)$ as $\eps\to 0.$ 
\end{itemize}
\end{theorem} 

\begin{proof} The first part of the result is proved in \cite[Theorem  2.1 and Corollary 2.4]{bcg}. 
%for connected open sets with Lipschitz boundary. 
%The result is obtained first of all for clusters defined in  $\R^d$, in \cite[Theorem 2.2]{bcg}, 
%and then applying an extension argument to deal with general domains.
By inspecting the proof in \cite{bcg} one can check that if the initial cluster is polyhedral outside $\Omega$, then the approximating sequence of polyhedral clusters $\mathcal{K}_\eps$ can be chosen in such a way that  $\mathcal{K}_\eps=\mathcal{F}$ in $\R^d\setminus \Omega^\eps$. 

We fix now $\delta>0$  sufficiently small and we substitute $\mathcal{F}$ 
in $\Omega\setminus \overline{\Omega_\delta}$ with the reflection of $\mathcal{F}$ from $\Omega^\delta\setminus\overline{\Omega}$.  The reflection is constructed as follows: We identify points in $\Omega^\delta\setminus\overline \Omega$ and points in $\Omega\setminus \overline{ \Omega_\delta}$  by putting  $x+t\hat \nu(x)=x-t\hat \nu(x)$ for $t\in (0, \delta)$, where 
 $\hat \nu(x)$ is a $C^1$ function which coincides on $\partial \Omega$  with  the outer normal at $x$. 
 In this way we obtain 
a new cluster $\mathcal{F}_\delta$  which coincides with $\mathcal{F}$ in $(\R^d\setminus\Omega)\cup \Omega_\delta$, 
and which is the reflection of $\mathcal{F}$ in $\Omega\setminus\overline{ \Omega_\delta}$. 
Note that, by construction, $\mathcal{F}_\delta$ is transversal to $\partial\Omega$. By using the previous result in the set $\Omega_\delta$, we construct  a family of approximating polyhedral clusters $\mathcal{K}_{\eps, \delta}$ for $\eps\to 0$, which coincide with $\mathcal{F}_\delta$ in $\R^d\setminus (\Omega_\delta)^\eps$. We choose now $\eps=\eps(\delta)<\delta$, 
 so that $(\Omega_\delta)^{\eps(\delta)}\subset\Omega$: Therefore 
$\mathcal{K}_{\eps(\delta), \delta}$ is a polyhedral cluster which coincides with $\mathcal{F}_\delta$ in $\R^d\setminus (\Omega_\delta)^{\eps(\delta)}$ and so in particular coincides with $\mathcal{F}$ in  $\R^d\setminus \Omega$, and   is transversal to $\partial \Omega$. 
Moreover   $\mathcal{K}_{\eps(\delta), \delta}\to \mathcal{F}$ in $L^1(\Omega)$ as $\delta\to 0$, and     for every $\eta>0$ sufficiently small,  there holds $\mathcal{P}_{c}(\mathcal{K}_{\eps(\delta), \delta}; \Omega_\eta)\to \mathcal{P}_{c}(\mathcal{F}; \Omega_\eta)$ as $\delta\to 0$. This implies the conclusion. 
 \end{proof}

 We now provide a $\Gamma$-convergence result, which is based on the analogous result obtained for the single  phase in  \cite{adpm,cv1} and by the density of polyhedral clusters in  Theorem \ref{dens}.    \begin{theorem}\label{teo:gamma} 
 Let $\Omega$  be   a $C^1$  bounded open set and let $\bar{ \mathcal{E}}$ a cluster which is polyhedral in $\R^d\setminus \Omega$ and is transversal to $\partial \Omega^+$.

For every sequence of positive numbers $c=(c_i)_i$,   as $s\to 1$ there holds 
\begin{equation}\label{eqGamma}
(1-s) \mathcal{P}_{s,c}(\mathcal{E}; \Omega)\stackrel{\Gamma}{\longrightarrow}\omega_{d-1}  \mathcal{P}_{c}(\mathcal{E}; \Omega),
\end{equation}
with respect to the $L^1(\Omega)$-convergence, where the functionals  $\mathcal{P}_{s,c}(\mathcal{E}; \Omega)$ and  $\mathcal{P}_{c}(\mathcal{E}; \Omega)$ are defined only on clusters $\mathcal{E}$ such that $\mathcal{E}=\bar{\mathcal{E}}$ in $\R^d\setminus \Omega$, and extended as $+\infty$ elsewhere. 
 
  \end{theorem} 
 
 \begin{proof} Let $s\to 1$, $\mathcal{E}^s, \mathcal{E}$ clusters which coincide with $\bar{\mathcal{E}}$ outside $\Omega$ and such that $\mathcal{E}^{s}\to \mathcal{E}$ in $L^1(\Omega)$.  Then using the $\Gamma-$liminf inequality for the single phase proved in \cite{adpm,cv1} we get  \begin{eqnarray}\label{inf} \nonumber
\liminf_{s\to 1}(1-s)\mathcal{P}_{s,c}(\mathcal{E}^{s};\Omega)&\geq& \sum_{i=1}^k c_i \liminf_{s\to 1}(1-s)\Per_{s}(E_i^{s};\Omega)
\\ 
&\geq& \omega_{d-1} \sum_{i=1}^k c_i\Per(E_i; \Omega)=\omega_{n-1}\mathcal{P}_{c}(\mathcal{E};\Omega). 
\end{eqnarray} 
Fix now a cluster $\mathcal{E}$ which coincides with $\bar{\mathcal{E}}$ outside $\Omega$.  By the $\Gamma$-liminf inequality we can restrict to consider clusters whose phases have finite perimeter in $\Omega$.
By Theorem \ref{dens}, for every $\eps$, there exist  polyhedral $\mathcal{K}_\eps$  which are transversal to $\partial\Omega$, coincide with $\bar{\mathcal{E}}$  in $\R^d\setminus\Omega$,  and satisfy $\mathcal{K}_\eps\to \mathcal{E}$ in $L^1(\Omega)$ and  
$\mathcal{P}_{c}(\mathcal{K}_\eps;\Omega)\to \mathcal{P}_{c}(\mathcal{E};\Omega)$ as $\eps\to 0$. 
 By \cite[Lemma 8]{adpm}, there holds  for all $\eps$  
\begin{eqnarray*}\nonumber&& 
\limsup_{s_n\to 1}(1-s_n)\mathcal{P}_{s_n,c}(\mathcal{K}_\eps;\Omega)\leq \sum_{i=1}^k c_i\limsup_{s_n\to 1}(1-s_n)\Per_{s_n}(K_\eps^i;\Omega)\\
& \leq&  \omega_{d-1}\sum_{i=1}^k c_i\Per(K_\eps^i; \Omega)=\omega_{d-1}\mathcal{P}_{c}(\mathcal{K}_\eps;\Omega)\leq \omega_{d-1}\mathcal{P}_{c}(\mathcal{E};\Omega)+o_\eps(1)\end{eqnarray*} where $o_\eps(1)\to 0$ as $\eps\to 0$. 
We conclude recalling that  $\mathcal{K}_\eps\to \mathcal{E}$ in $L^1(\Omega)$ as $\eps\to 0$, and choosing $\eps_n=\eps(s_n)\to 0$ as $s_n\to 1$. 
 \end{proof}

Finally, using the $\Gamma$-convergence result and the density estimates  recalled in Theorem \ref{density}, which are uniform in $s\geq s_0$, we get uniform convergence of minimizers of the Dirichlet problem as $s\to 1$ to the  minimizer of the  Dirichlet problem with local perimeter. 

We recall the definition of Hausdorff convergence.

 \begin{definition} 
 Let $E_n,  E\subset\Omega$, where $\Omega$ is a   open set. We say that $E_n\to E$ locally uniformly in $\Omega$, if for any $\eps>0$ and any $\Omega'\subset \subset \Omega$,  there exists $\bar n$ such that  for all $n\geq \bar n$, we have that 
\[
 \sup_{x\in E_n\cap \Omega'} d(x, E)\leq \eps\quad \text{and} \quad\sup_{x\in (\Omega\setminus E_n)\cap \Omega'} d(x, \Omega\setminus E)\leq \eps.
\]
 \end{definition} 
First of all we state an   equicoercivity property of the functionals $(1-s) \mathcal{P}_{s,c}$, which is obtained by applying to each phase    the equicoercivity result  in \cite[Theorem 1]{adpm}.  
\begin{lemma} \label{lemmacoe} Let $\Omega$  be  a  bounded open set,  $\bar{ \mathcal{E}}$ a $k$-cluster   in $\R^d\setminus \Omega$ and $c=(c_i)$ a sequence of positive numbers. 
Let $s_n\to 1$, and $\mathcal{E}^{s_n}$ a family of $k$-clusters with  $\mathcal{E}^{s_n}=\bar{ \mathcal{E}}$ in $\R^d\setminus \Omega$ and with equibounded energy, that is there exists $C>0$ for which 
\[\sup_{n} (1-s_n) \mathcal{P}_{s_n,c}(\mathcal{E}^{s_n} \Omega)\leq C.\] Then $\mathcal{E}^{s_n}$ is relatively  compact in $L^1(\Omega)$.
\end{lemma} 
 
\begin{theorem} \label{convthm} 
Under the assumptions of Theorem \ref{teo:gamma}, let $s_n\to 1$ and let 
$\mathcal{E}^{s_n}= (E_{1}^{s_n},\dots, E_k^{s_n}) $  be a sequence of minimizers of 
\begin{equation}\label{caio}
\inf_{\{\mathcal{F},\  F_i\setminus \Omega=\bar E_i\}} \mathcal{P}_{s_n,c}(\mathcal{F}; \Omega).\end{equation}
Then, up to a subsequence, $E_i^{s_n}\to E_i$ locally  uniformly in $\Omega$, where  $\mathcal{E}=(E_1, \dots, E_k)$ is a minimizer of \[ \inf_{\{\mathcal{F},\   F_i\setminus \Omega=\bar E_i\}}\mathcal{P}_{c}(\mathcal{F}; \Omega).\]
Moreover, for any  $x\in \mathcal{R}(\mathcal{E})\cap \Omega$  there exists $r_x>0$  such that
$\partial \mathcal{E}^{s_n}\cap B(x, r_x)$  is $C^\infty$-diffeomorphic to $\partial \mathcal{E}\cap B(x, r_x)$ for $n$ large enough.
\end{theorem}

\begin{proof}First of all, we observe that due to  minimality, reasoning as in the proof of Theorem \ref{exth},  $ (1-s_n) \mathcal{P}_{s_n,c}(\mathcal{E}^{s_n} ; \Omega)\leq  k \max c_i  (1-s_n)\Per_{s_n}(\Omega)\leq C$, since $\lim_n   (1-s_n)\Per_{s_n}(\Omega)= \Per(\Omega)$, see  \cite{cv1}. Now, by Lemma \ref{lemmacoe}, up to passing to a subsequence we have that $\mathcal{E}^{s_n}\to \mathcal{E}$ in $L^1(\Omega)$ and by Theorem \ref{teo:gamma},  $\mathcal{E}=(E_1, \dots, E_k)$ is a minimizer of \[ \inf_{\{\mathcal{F},\   F_i\setminus \Omega=\bar E_i\}}\mathcal{P}_{c}(\mathcal{F}; \Omega).\]

 We show now that, by the density estimates in Theorem \ref{density},  we get that the convergence is locally uniform in $\Omega$. Assume by contradiction that it is not true. 
Then, for some $\Omega'\subset \subset \Omega$ and  for some $\eps>0$, either there exists $x_k\in E_i^{s_k}\cap \Omega'$ such that $d(x_k, E_i)>\eps$ for all $k$  or there exists  $x_k\in (\Omega\setminus E_i^{s_k})\cap \Omega' $ such that $d(x_k, \Omega\setminus E_i)>\eps$. Let us consider the first case (the second is completely analogous). 
By the density estimates in Theorem \ref{density},  letting  $2\delta=\min(d(\partial \Omega', \partial \Omega),\eps)$ we get that $|E_i^{s_k}\cap B(x_k,\delta)|\geq \sigma_0\omega_n \delta^{n}$ for all $k$. Note that $A_k:=E_i^{s_k}\cap B(x_k,\delta)\subset\subset \Omega$, $|A_k|>c>0$ uniformly in $k$ and $A_k\cap E_i=\emptyset$, in contradiction with the $L^1(\Omega)$-convergence of $\chi_{E_k}$ to $\chi_E$.

Finally, let us fix a regular point $x\in \partial \mathcal{E}\cap \Omega$. Then, there exist two indexes $i,j$  and $r>0$ such that $E_h\cap B(x,2r)=\emptyset $  for all $h\neq i,j$.  By Hausdorff convergence, there exists $n_0$ such that for $n>n_0$ there holds that $E_h^{s_n}\cap B(x,r)=\emptyset$ for all $h\neq i,j$ and moreover, reasoning as in the proof of Theorem \ref{thmsing} $E_i^{s_n}$ is a $\Lambda$-minimizer for $\Per_{s_n}$ in $B(x,r)$, where $\Lambda$ can be chosen uniform in $n>n_0$.  By the uniform in $s$ improvement of flatness of $\Lambda$-minimizers of $\Per_s$  proved in \cite[Theorem 3.4, Corollary 3.5]{ffmmm}, we get that, eventually reducing $r$, all the points in  $\partial\mathcal{E}^{s_n}\cap B(x,r)$ are regular for $n>n_0$. Finally, by \cite[Corollary 3.6]{ffmmm} we conclude that  there exist $\alpha\in (0,1)$ and a sequence $\psi_{s_n}\in C^{1,\alpha}(\partial E_i\cap B(x,r))$  such that $\|\psi_{s_n}\|_{C^{1,\alpha}}\leq C$ for $n>n_0$, $\lim_{s_n\to 1}\|\psi_{s_n}\|_{C^{1}}=0$ and $ \partial E_i^{s_n}\cap B(x,r)=(Id+\psi_{s_n}\nu_{E_i})(\partial E_i\cap B(x,r))$, for all $n>n_0$. Actually, by the bootstrap argument in \cite[Theorem 6]{bfv} actually $\psi_{s_n}\in C^\infty(\partial E_i\cap B(x,r))$, with uniform norm.    This gives the conclusion. \end{proof} 

\begin{remark}\upshape \label{remsin}
Note that Theorem \ref{convthm} does not imply that 
$\partial \mathcal{E}^{s_n}\cap\Omega$ is diffeomorphic to $\partial \mathcal{E}\cap\Omega$ for $n$ large enough. 
The main obstruction to obtain such a result (which is expected) is the lack of a regularity theory up to the singular set of the cluster.
We point out that, for cluster minimizing the classical perimeter, the regularity theory around singular points is well-developed only in dimension $d=2,3$ 
(see \cite{maggibook,cm}).    
\end{remark} 

\begin{remark}\upshape We observe that all the results in this section can be easily extended to the isoperimetric clusters considered in \cite{cm}.
\end{remark}

%%%%%%%%%%%%%%%%%%%%%%%%%%%%%%%%%%%%%%%%%%%%%%%%  
\section{Minimal cones} 
%%%%%%%%%%%%%%%%%%%%%%%%%%%%%%%%%%%%%%%%%%%%%%%%
In this section we restrict to the 2-dimensional case, $d=2$, and to consider the functional \eqref{fun2}, that is we assume that all the weights $c_i$ are equal.

We recall the definition of local minimizer (or minimizer up to compact perturbations).
\begin{definition} We say that the $k$-cluster $\mathcal{E}$ is a local minimizer for \eqref{fun2} if for every $R>0$ and every ball $B_R$ of radius $R$, there holds
\[\mathcal{P}_s(\mathcal{E}; B_R)\leq \mathcal{P}_s(\mathcal{F}; B_R)\]
for all  $k$-clusters $\mathcal{F} $, such that $F_i\setminus B_R= E_i\setminus B_R$ for all $i$. 
\end{definition}

We now observe that there exists a unique  $3$-cone which is a stationary point for \eqref{fun2}. 

\begin{lemma}  \label{lemma3cones}
Among all $3$-cones in $\R^2$, 
 there exists a unique cone which is stationary for the functional in \eqref{fun2}, and the opening angles are equals, and coincide with $2/3\pi$.
\end{lemma}

\begin{proof} We consider a cone $\mathcal{C}=(C_1, C_2, C_3)$ with 3 half-lines and vertex $x_0$ which is stationary for the functional \eqref{fun2} (so, the first variation of \eqref{fun2} at every boundary point is $0$). We denote with $\alpha_i$ the angle associated to the sector $E_i$, so $\alpha_1+\alpha_2+\alpha_3=2\pi$.  
Up to a translation we assume that the vertex of the cone is $0$. 

The stationarity condition reads 
\begin{equation}\label{cond}H_s(x, C_i) =H_s(x, C_j) \qquad \forall x\in  \partial C_i\cap \partial C_j,\  x\neq 0\end{equation}
where $H_s(x, C_i)$ is the fractional curvature at $x\in \partial C_i$, defined in \eqref{fcurv}. 

It is easy to check that of $x\in \partial C_i\cap \partial C_j$, we have that \begin{equation}\label{condp}H_s(x, C_i)\leq 0\qquad \text{ if and only if }\quad \alpha_i\geq \pi.\end{equation}   Using this observation, \eqref{cond},  and the fact that $\alpha_1+\alpha_2+\alpha_3=2\pi$, we  have that  $\alpha_i<\pi$.   

We exploit now condition \eqref{cond} for $i=1, j=2$ (all the other cases will be analogous). 
We assume without loss of generality that $\alpha_1\geq \alpha_2$ and we write $C_1= \tilde C_2 \cup B$, where $\tilde C_2$ is the symmetric of $C_2$ with respect to the half-line separating $C_1, C_2$ and $B$ is a sector of the cone with opening angle $\alpha_1-\alpha_2$.  Let $\tilde B\subseteq C_3$ be the symmetric of $B$ with respect to the half-line separating $C_1, C_2$. 
By symmetry properties of the kernel  it is easy to check that
\begin{eqnarray}\label{c} 
H_s(x, C_1)&=& \int_{C_3}\frac{1}{|x-y|^{2+s}}dy-\int_B \frac{1}{|x-y|^{2+s}}dy=\int_{(C_3\setminus \tilde B)-x}\frac{1}{|y|^{2+s}}dy,\\H_s(x, C_2) &=&\int_{C_3}\frac{1}{|x-y|^{2+s}}dy+\int_{B} \frac{1}{|x-y|^{2+s}}dy=\int_{(C_3\cup B)-x}\frac{1}{|y|^{2+s}}dy.\nonumber 
\end{eqnarray} 
Note that $C_3\setminus \tilde B$  is a sector of the cone with opening angle $\alpha_3-\alpha_1+\alpha_2= 2\pi-2\alpha_1>0$, whereas  $C_3\cup  B$  is a sector of the cone with opening angle 
$\alpha_3+\alpha_1-\alpha_2=2\pi-2\alpha_2>0$, and both are symmetric with respect to the half-line separating $C_1, C_2$. 
Therefore condition \eqref{cond} implies that $2\pi-2\alpha_1=2\pi-2\alpha_2$. 
Repeating the argument we get that $\alpha_1=\alpha_2=\alpha_3$. 
\end{proof} 

 \begin{proposition}\label{dirdue} 
 Let $\Omega\subset\R^2$ be a bounded  $C^1$ open set containing the origin, let $k=3$  and let $\bar E_i$ be the exterior datum defined as 
 \[
 \bar E_i:= \left\{x\in\R^2:\,x\cdot n_i>\frac 12\right\},
 \qquad \quad n_i := \left( \cos\left(\frac 23 \pi i\right),\sin\left(\frac 23 \pi i\right)\right).
 \]
Then there exists $s_0\in (0,1)$  such that for $s>s_0$ every minimizer of the Dirichlet problem \eqref{dir2} with $c_i=1$ for all $i$   
has a nonempty singular set in $\Omega$.
 \end{proposition} 
 
 \begin{proof}  
 Let $\mathcal{E}_s=(E_1^s, E_2^s, E_3^s)$ be a solution to the Dirichlet problem \eqref{dir2} with $c_i=1$. Let $\mathcal{E}$ the solution to the  Dirichlet problem  with  the same boundary data and functional given by the local perimeter \eqref{wpc}, with all $c_i=1$. Then $\mathcal{E}$ is the solution of the classical geometric Steiner problem and $E_i=\bar E_i$ for every $i$.  By Theorem \ref{teo:gamma}, up to a subsequence we get that $E_i^s\to \bar E_i$ locally uniformly in $\Omega$ as $s\to 1$, for $i\in \{1,2,3\}$.
 Let $R>0$ be such that $B(0,R)\subset\Omega$.
 
 Assume by contradiction that there is a sequence $s_n\to 1$ such that $\partial E_i^{s_n}\cap\Omega$ is of class $C^1$ for all $n$'s. 
 There exists $r\in (0,R)$ such that, 
 for $i\ne j$, the set $\gamma^n_{ij}:=\partial E_i^{s_n}\cap \partial E_j^{s_n}\cap B(0,r)$ is a finite number of $C^1$ curves with
 endpoints on $\partial B(0,r)$, converging to the segment 
 $\partial \bar E_i\cap \partial\bar E_j\cap B(0,r)$ as $n\to +\infty$ in the Hausdorff distance. In particular, given $\eps>0$, 
 for $n$ large enough the set $\gamma^n_{ij}$ 
 divides the circle $B(0,r)$ into a finite number of small connected components and one large connected component of area greater than
 $|B(0,r)|-\eps$. As a consequence either the set $E_i^{s_n}\cap B(0,r)$ or  $E_j^{s_n}\cap B(0,r)$
 is contained in the union of such small connected components,
 so that either $|E_i^{s_n}\cap B(0,r)|\le \eps$ or  $|E_j^{s_n}\cap B(0,r)|\le \eps$ for $n$ large enough,
 contradicting the convergence of $E_k^{s_n}\cap B(0,r)$ to $\bar E_k\cap B(0,r)$, for all $k\in\{1,2,3\}$.
 \end{proof}

\begin{theorem} \label{cone} 
There exists $s_0\in (0,1)$ such that the following holds: Among all cones, the unique local minimizers for $\mathcal{P}_s$, for $s>s_0$, are  half-planes and $3$-cones with equal  opening angles given by  $2/3 \pi$. 
\end{theorem} 

\begin{proof}
Let $s_n\to 1$ and let   $\mathcal C_n$ be a sequence of minimal cones for $\mathcal P_s$.
By Theorem \ref{teo:gamma} there exists a minimal cone $\mathcal C$ for the classical perimeter such that 
$\mathcal C_n\to \mathcal C$ locally uniformly as $n\to \infty$. Since the only minimal cones in $\R^2$ are half-planes
or $3$-cones with angles of $2/3\pi$ \cite{alm}, it follows by the uniform convergence
 that also the $\mathcal C_n$'s are a half-spaces or $3$-cones for $n$ large enough.
 By Lemma \ref{lemma3cones}, if $\mathcal C_n$ is a minimal $3$-cone then necessarily 
 it has equal angles of $2/3\pi$. 
 
 By Proposition \ref{dirdue}  we know that there exist minimal cones which are not half-planes,
 and this concludes the proof.
 \end{proof}
 
 \begin{remark}\upshape An interesting issue which is left open is whether Theorem \ref{cone} is true for all $s\in (0,1)$. We conjecture this is the case, 
 but in order to prove this result it would be necessary to develop some new technical argument. A related problem is about the possibility  of extending the nonlocal calibrations recently introduced in \cite{c,p} to clusters, 
 in the same spirit of the paired calibrations used in \cite{lf}. 
 \end{remark}

\begin{remark}\upshape  
By Theorem \ref{convthm},  for every $r>0$ there exists $s_r\in (0,1)$ such that the  solution to the Dirichlet problem given in Proposition \ref{dirdue},
with $s\in [s_r,1)$, is diffeomorphic in $\Omega\setminus B(0,r)$ to the solution of the classical Steiner problem, which is given by $(\bar E_1, \bar E_2, \bar E_3)$.   We point out, recalling Remark \ref{remsin}, that even if the limit cluster has only one singular point in $0$, 
our results do not exclude that the approximating clusters have more singular points, all converging to $0$ as $s\to 1$.  
 \end{remark}

\section{Weighted fractional perimeters}
Let us fix a sequence $c_i$ with $i\in\N$, such that $c_i>0$ for all $i$ and consider the energy associated to a $k$-cluster $\mathcal{E}$ and to the sequence $c_i$ as
\begin{equation}\label{fun3} \mathcal{P}_{s,c}(\mathcal{E}; \Omega)=\sum_{1\leq i\leq k}c_i \Per_s (E_i; \Omega).
\end{equation}

First of all we consider the generalization of Lemma \ref{lemma3cones}.
\begin{lemma}   \label{lemma3cones2} 
Among all $3$-cones in $\R^2$ 
 there exists a unique cone which is stationary for the functional in \eqref{fun3}, and the opening angles are uniquely determined as functions of $c_i$. 
  \end{lemma}
\begin{proof} The proof is analogous to that of Lemma \ref{lemma3cones}. 
The stationarity condition reads 
\begin{equation}\label{condnuova}c_i H_s(x, C_i) =c_jH_s(x, C_j) \qquad \forall x\in  \partial C_i\cap \partial C_j,\  x\neq 0,\end{equation}
and since $c_i>0$ for all $i$, we get $\alpha_i<\pi$.   

Proceeding as in \eqref{c} in the proof of Lemma \ref{lemma3cones}  and using the same notation, we  note that for all $\lambda>0$,  $\lambda((C_3\setminus \tilde B)-x)=(C_3\setminus \tilde B)-\lambda x$ and $\lambda((C_3\cup B)-x)=(C_3\cup B)-\lambda x$. Therefore 
$H_s(x, C_i)=\lambda^{s} H_s(\lambda x, C_i)$. This implies that it is sufficient to verify condition \eqref{condnuova} just  for one $x\neq 0$.
We fix from now on $x$, with $|x|=1$.

We introduce the function $F:[0, \pi)\to \R$ as 
\begin{equation}\label{falpha}
F(\alpha)= 2\int_0^\alpha \int_0^{+\infty} \frac{\rho}{(1+\rho^2+2\rho \cos\theta)^{1+s/2}}d\rho d\theta.
\end{equation}    
Note that if $K$ is a sector of the cone with opening angle $2\alpha$ and which is symmetric with respect to the half-line separating $C_1, C_2$,  then
$F(\alpha)= \int_{K}\frac{1}{|x-y|^{2+s}}dy$.  Note that $F(0)=0$ and 
\[F'(\alpha)= 2  \int_0^{+\infty} \frac{\rho}{(1+\rho^2+2\rho \cos\alpha)^{1+s/2}}d\rho>0.\] Therefore $F$ is  invertible. 

Recalling the definition of $F$ and \eqref{c}, we may  restate \eqref{condnuova} as 
\begin{equation}\label{cond2}c_2 F(\pi-\alpha_2) =c_1F(\pi-\alpha_1).\end{equation}

With the same argument we conclude that the cone $\mathcal{C}$ is stationary iff 
\begin{equation}\label{stat}c_2 F(\pi-\alpha_2) =c_1F(\pi-\alpha_1)=c_3F(\pi-\alpha_3).\end{equation}

Let $k>0$ be the solution to the equation
\[F^{-1}(k/c_1)+F^{-1}(k/c_2)+F^{-1}(k/c_3)=\pi,\] 
which exists and is unique due to the fact that $F^{-1}:[0, +\infty)\to \R$ is monotone increasing. 
Then  the angles $\alpha_i$ are uniquely determined as
\[\alpha_i= \pi-F^{-1}(k/c_i). \]

\end{proof} 
\begin{remark}\label{clas}
In the case of standard perimeter, it has been proved in \cite{l} that the unique $3$-cone which is a local minimizer for the functional $\sum_{1\leq i\leq 3}c_i \Per (E_i)$ has opening angles $\alpha_i$ which satisfies the following relation
\[\frac{\sin\alpha_1}{c_2+c_3}=\frac{\sin\alpha_2}{c_1+c_3}=\frac{\sin\alpha_3}{c_1+c_2}.\]
For general $k$-clusters, with $k>3$, in general there could be singular cones with more than $3$ phases which are local minimizers. 
%For example in  \cite{dn}  it is shown that the cone with  four equal squares with a vertex in common is a  local minimizer if we consider  weights $>\sqrt{2}$  between diagonally-opposite squares and $1$  otherwise; a similar example is also considered in \cite[Section 2.3]{ab}.
 However, in \cite{dn} it is proved that  if the weights $c_i$ are sufficiently close to $1$, it is possible to recover  the triple-point property: Only $3$-cones are local minimizers.
\end{remark}
We get in this case the following  analogous of Theorem \ref{cone} for the case of $3$ cones. We state it in this form since for the functional  $\sum_{i}c_i \Per (E_i)$  it is not known if the unique local minimizers among cones are just  half-planes and the $3$-cone given in Remark \ref{clas}, see \cite{lf}. 
\begin{proposition} \label{cone2} 
There exists $s_0\in (0,1)$  depending on $(c_i)_i$ such that the following holds: 
Among all $2$-cones and $3$-cones,
 the unique local minimizers for $\mathcal{P}_{s,c}$, for $s>s_0$, are half-planes and the  $3$-cone obtained in Lemma \ref{lemma3cones2}. 
\end{proposition} 

\begin{proof} Arguing as in the proof of Theorem \ref{cone}, we consider  $s_n\to 1$ and $\mathcal{C}_n$ to be a sequence of minimal cones for 
$\sum_{i=1}^3 c_i \Per_{s_n}(\cdot)$.
By Theorem \ref{teo:gamma} there exists a minimal cone $\mathcal C$ for $\sum_{i=1}^3 c_i \Per(\cdot)$ 
 such that 
$\mathcal C_n\to \mathcal C$ locally uniformly as $n\to \infty$. Since the only minimal cones in $\R^2$ are half-planes
or $3$-cones with angles  given in Remark \ref{clas}, it follows by the uniform convergence
 that also the $\mathcal C_n$'s are a half-planes or $3$-cones for $n$ large enough.
 By Lemma \ref{lemma3cones2}, if $\mathcal C_n$ is a minimal $3$-cone then necessarily 
it coincides with the $3$-cone computed in the Lemma.  
Arguing as in Proposition \ref{dirdue}, and recalling Remark \ref{clas},  we get  that there exist minimal cones which are not half-planes,
 and this concludes the proof.
\end{proof}

%\addcontentsline{toc}{section}{References}
%\bibliographystyle{plain}
\bibliography{fractional}
\bibliographystyle{abbrv}
\end{document}